\documentclass[11pt]{amsart}
\usepackage{pdfsync}
\usepackage{fullpage}
\usepackage{amsmath}
\usepackage{amssymb}
\usepackage{MnSymbol}
\usepackage{verbatim}
\usepackage[linktocpage]{hyperref}
\usepackage{enumitem}
\usepackage{hyperref}

\usepackage{graphicx}
\usepackage{subfigure}
\usepackage{xcolor}

\newtheorem{df}{Definition}[section]
\newtheorem{thm}[df]{Theorem}

\newtheorem{rem}[df]{Remark}

\newtheorem{lem}[df]{Lemma}

\newcommand{\ga}{\gamma}

\newcommand{\mb}{\mathbb}
\newcommand{\ra}{\rightarrow}
\newcommand{\pr}{^{\prime}}
\newcommand{\eq}{eqnarray*}
\newcommand{\ep}{\epsilon}
\newcommand{\tri}{\triangle}
\newcommand{\hyp}{hyperbolic }

\title{A Remark on the Rigidity of Delaunay Triangulated Plane}
\author{Song Dai\textsuperscript{1}}

\address{Song Dai\\
Center for Applied Mathematics\\
Tianjin University\\
No.92 Weijinlu Nankai District\\
Tianjin\\
P.R.China 300072}

\email{song.dai@tju.edu.cn}
\date{}
\begin{document}

\begin{abstract}
In \cite{Wu22}, under the uniformly acute condition, Wu showed the rigidity of the geodesic triangulated plane under Luo's discrete conformality. In this article, by modifying Wu's proof, we improve this result by weakening the uniformly acute condition to the uniformly Delaunay condition.
\end{abstract}

\maketitle

\section{Introduction}
Let $T=(V,E,F)$ be a triangulation of a surface $S$ with or without boundary. Denote $|T|$ as the underlying space of the complex $T$. A PL (piecewise linear) metric on $T$ is a function, $l:E\rightarrow \mb{R}_+$ such that for every $ijk\in F$, $ijk$ forms a Euclidean triangle under the length $l$, denoted as $\triangle ijk$.
In \cite{Luo04}, Luo introduced the notion of the discrete conformality. Let $l,l\pr$ be two PL metrics on $T$. We call $l$ is discrete conformal to $l\pr$ if there exists a function $u: V\ra \mb{R}$ such that $$l_{ij}=e^{(u_i+u_j)/2}l\pr_{ij},$$
denoted as $l\pr=u*l$. The function $u$ is called the conformal factor.

In recent years, the theory of Luo's discrete conformality has developed in many directions, such as the prescribed curvature problems, the rigidity, the convergence, the numerical methods in computing conformal geometry and so on. One may refer to \cite{BPS15,CCSZPZ21,DGM22,GGLSW18,GLSW18,GLW19,GSC21,LSW20,LW19,LWZ21a,Spr19,SWGL15,WGS15,Wu14,Wu22,WX21,WZ20}. In this article, we focus on the rigidity of the triangulated plane.

For a triangle $\triangle ijk$, denote $\theta_{jk}^i$ as the angle of $\angle jik$. A PL metric $l$ is called\\
(1) uniformly nondegenerate if there is a constant $\ep>0$ such that $\theta^i_{jk}\geq \ep$ for every triangle $\triangle ijk\in F$;\\
(2) Delaunay if $\theta^k_{ij}+\theta^l_{ij}\leq \pi$ for every adjacent triangles $\tri ikj,\tri ilj\in F$ sharing the edge $ij\in E$;\\
(3) uniformly Delaunay if there is a constant $\ep>0$ such that $\theta^k_{ij}+\theta^l_{ij}\leq \pi-\ep$ for every adjacent triangles $\tri ikj,\tri ilj\in F$ sharing the edge $ij\in E$;\\
(4) uniformly acute if there is a constant $\ep>0$ such that $\theta^i_{jk}\leq \frac{\pi}{2}-\ep$ for every triangle $\triangle ijk\in F$.
\begin{rem}
The Delaunay condition is equivalent to that for every adjacent triangles $\tri ijk,\tri ijl\in F$ isometric embedding in $\mb{C}$, then $l\notin \text{int}(C)$, where $\text{int}(C)$ is the interior of the circumscribed circle of $\tri ijk$.
\end{rem}
\begin{rem}
It is clear that under the uniformly nondegenerate condition with constant $\ep>0$, $\deg i$ are uniformly bounded, $\deg i\leq \frac{2\pi}{\ep}, \forall i\in V.$
\end{rem}

Denote $|T|^o$ and $\partial |T|$ be the interior and the boundary of $|T|$ respectively. Denote $V^o=V\cap |T|^o$, $\partial V=V\cap \partial |T|$. Let $l$ be a PL metric on $T$. For $i\in V^o$, denote $R_i$ as the $1$-ring neighborhood of $i$, that is the subcomplex generated by $i$ and $j$ for every $ij\in E$. We abuse the notation $R_i$ also as the underlying space $|R_i|$ and the set of vertices $V(R_i)$.
The curvature at $i$ is defined as
$$K_i=2\pi-\sum\limits_{\tri ijk\in F}\theta^i_{jk},$$ which only depends on the restriction of $l$ on $R_i$.
Let $l\pr=u*l$ be also a PL metric discrete conformal to $l$. Then $K_i(u):\mb{R}^{\deg i+1}\ra \mb{R}$ is smooth with respect to $u$. From the direct calculation or \cite{Luo04}, one have
$$dK_i=-\sum\limits_{j:ij\in E}\eta_{ij}(du_j-du_i),$$
where $$\eta_{ij}=\eta_{ij}(u)=\frac{1}{2}(\cot \theta^{k}_{ij}(u)+\cot \theta^{l}_{ij}(u))$$
for adjacent triangles $\tri kij,\tri lij\in F$ sharing the edge $ij\in E$.
It is clear that under the Delaunay condition $\eta\geq 0$, and under the uniformly Delaunay condition $\eta\geq \ep>0.$ We call $l$ is flat if $K_i=0$ for every $i\in V$.
\begin{rem}
For $i\in\partial V$, the curvature is defined as $K_i=\pi-\sum\limits_{jk:\tri ijk\in F}\theta^i_{jk}.$ It is clear that $K$ is invariant under dilation, i.e. $u$ being a constant.
\end{rem}

A map $\phi:|T|\ra \mb{C}$ is called a geodesic embedding if for every $ij\in E$, $\phi$ maps $ij$ to a segment connecting $\phi(i)$ and $\phi(j)$, and $\phi$ maps $|T|$ homeomorphically to its image. If further $\phi$ is surjective, we call $\phi$ is a geodesic homeomorphism or a geodesic triangulated plane. It is clear that a geodesic embedding $\phi$ gives a flat PL metric $l(\phi)$, or $l$ for short, by using the Euclidean distance.

Let $\phi,\phi\pr$ be two geodesic triangulated planes with PL metric $l,l\pr$. Suppose $l,l\pr$ are discrete conformal, $l\pr=u*l$. The conformal factor $u$ being a constant is clearly a solution. The rigidity problem is that whether $u$ must be a constant. In \cite{WGS15}, Wu-Gu-Sun first gave an affirmative answer to this problem under the condition $T$ being the standard hexagonal triangulation of the plane with $l\equiv 1$ and $l\pr$ satisfying the uniformly acute condition. They also gave a counter-example for the case $l\pr$ being flat but not a geodesic triangulated plane. In \cite{LSW20} Luo-Sun-Wu improved the result above by showing that the rigidity of the standard hexagonal geodesic triangulated plane holds under the condition $l\pr$ satisfying the Delaunay condition. Dai-Ge-Ma also showed this result in \cite{DGM22} independently. Recently, Wu in \cite{Wu22} made a breakthrough to release the special case of the geodesic triangulated plane to general cases under certain conditions.
\begin{thm}\label{main acute}(Theorem 1.2 in \cite{Wu22})
Let $T$ be a triangulated plane. Let $\phi:|T|\ra \mb{C}$ be a geodesic homeomorphism with the induced PL metric $l$. Let $\phi\pr:|T|\ra \mb{C}$ be a geodesic embedding with the induced PL metric $l\pr$. Suppose both $l,l\pr$ satisfy the uniformly nondegenerate condition and the uniformly acute condition. If $l$ and $l\pr$ are discrete conformal, $l\pr=u*l$, then $u$ is a constant.
\end{thm}
In this article, we modify Wu's proof in \cite{Wu22} to improve this result by weakening the uniformly acute condition to the uniformly Delaunay condition.
\begin{thm}\label{main Delaunay}
Let $T$ be a triangulated plane. Let $\phi:|T|\ra \mb{C}$ be a geodesic homeomorphism with the induced PL metric $l$. Suppose $l$ satisfies the uniformly nondegenerate condition and the uniformly Delaunay condition. Let $\phi\pr:|T|\ra \mb{C}$ be a geodesic embedding with the induced PL metric $l\pr$. Suppose $l\pr$ satisfies the uniformly nondegenerate condition and the Delaunay condition. If $l$ and $l\pr$ are discrete conformal, $l\pr=u*l$, then $u$ is a constant.
\end{thm}
\begin{rem}
In \cite{BPS15}, Bobenko-Pinkall-Springborn observed the relation between Luo's discrete conformality and the \hyp polyhedra in $\mb{H}^3$. In fact the rigidity problem in this article corresponds to the Cauchy rigidity of the ideal \hyp polyhedra and the Delaunay condition corresponds to the convexity of the \hyp polyhedra. One may also refer to \cite{DGM22}.
\end{rem}
For the study of geodesic triangulations of other surfaces, one may refer to \cite{Luo22,LWZ21b,LWZ21c,LWZ22,LZ22}.

The article is organized as follows. In Section \ref{hyperbolic}, we study the hyperbolic discrete conformality. In Section \ref{maximum}, we show some useful maximum principles in the discrete conformal geometry. In Section \ref{modification}, we modify Wu's proof to show the main result Theorem \ref{main Delaunay}.

\noindent\textbf{Acknowledgement:} The author would like to thank Tianqi Wu and Huabin Ge. The author is supported by NSF of China (No.11871283, No.11971244 and No.12071338).

\section{Hyperbolic discrete conformality}\label{hyperbolic}
The notion of the hyperbolic discrete conformality was first introduced by Bobenko-Pinkall-Springborn in \cite{BPS15}. A hyperbolic PL metric on $T$ is a function, $l^h:E\rightarrow \mb{R}_+$ such that for every $ijk\in F$, $ijk$ forms a hyperbolic triangle under the length $l^h$. Let $l^h,l^{h\prime}$ be two hyperbolic PL metrics on $T$. We call $l^h$ is \hyp discrete conformal to $l^{h\prime}$ if there exists a function $u^h: V\ra \mb{R}$ such that $$\sinh \frac{l^h_{ij}}{2}=e^{(u^h_i+u^h_j)/2}\sinh \frac{l^{h\prime}_{ij}}{2},$$ denoted as $l^{h\prime}=u^h*^hl^h$. The function $u^h$ is called the \hyp conformal factor.
A \hyp PL metric is called Delaunay if for every adjacent \hyp triangles $\tri ijk,\tri ijl\in F$ isometric embedding in the Poincar\'e disc $\mb{D}$, then $l\notin \text{int}(C)$, where $\text{int}(C)$ is the interior of the circumscribed circle of $\tri ijk$ in the Euclidean sense.
The curvature is similarly defined by using the \hyp metric. A map $\phi^h:|T|\ra \mb{C}$ is called a \hyp geodesic embedding if for every $ij\in E$, $\phi^h$ maps $ij$ to a \hyp geodesic segment connecting $\phi^h(i)$ and $\phi^h(j)$, and $\phi^h$ maps $|T|$ homeomorphically to its image. It is clear that a \hyp geodesic embedding $\phi^h$ gives a flat hyperbolic PL metric $l^h(\phi^h)$, or $l^h$ for short, by using the \hyp distance $d^h$, $l^h_{ij}=d^{h}(\phi^h(i),\phi^h(j))$.

Let $i\in V^o$ and $R_i$ be its $1$-ring neighborhood. Denote $D$ as the unit disc with the Euclidean metric and $\mb{D}$ as the unit disc with the \hyp metric. Let $\phi:R_i\ra D\subset \mb{C}$ be a geodesic embedding. It seems that $\phi$ also induces a \hyp geodesic embedding $\phi^h:R_i\ra\mb{D}$ by connecting the edges by the \hyp segments of $\mb{D}$ and mapping the faces homeomorphically to its image. It is not true in general, whose reason will be clarified in the proof of the following lemma, while it indeed holds under certain conditions.
\begin{lem}\label{hyp emb}
Let $\phi:R_i\ra D\subset \mb{C}$ be a geodesic embedding such that $\theta_{ij}^k\geq \ep$ for all $\tri ijk\in R_i$, for some constant $\ep>0.$ Suppose
$$l_{ij}<(1-|\phi(i)|^2)\sin\ep,\quad\text{ for every }ij\in E.$$
Then $\phi$ induces a \hyp geodesic embedding $\phi^h:R_i\ra\mb{D}$ such that $\phi^h$ coincides with $\phi$ on the set of vertices, maps the edges to the \hyp segments and maps the faces homeomorphically to its image.
\end{lem}
\begin{proof}
Let $\deg i=m$. Let $z_k=\phi(j_k)$, $j_k\in R_i\setminus\{i\}$, $k=1,\cdots,m$, be anti-clockwise on $\phi(R_i)$, and $z_0=\phi(i)$.
Let $\exp_{z_0}$ be the exponential map of the hyperbolic metric at $z_0$. Identifying $T_{z_0}\mb{D}$ with $\mb{C}$ by translating $z_0$ to the origin, then for $z\in\mb{D}$, denote $v(z)=\exp_{z_0}^{-1}z\in \mb{C}$.
Then we only need to show the following claims.
\begin{equation}\label{1}
\arg (\frac{v(z_{k+1})}{v(z_k)})\in(0,\pi),\quad k=1,\cdots,m,
\end{equation}
and
\begin{equation}\label{2}
\sum\limits_{k=1}^{m}\arg\big(\frac{v(z_{k+1})}{v(z_k)}\big)=2\pi,
\end{equation}
where $z_{m+1}=z_1$.
Fix $k\in \{1,\cdots,m\}$, denote $$P=\{z\in\mb{C}:\arg(\frac{z-z_0}{z_k-z_0})\in(0,\pi)\},$$
and $$P_h=\{z\in\mb{D}:\arg(\frac{v(z)}{v(z_k)})\in(0,\pi)\}.$$
\begin{figure}[htbp]\centering
\includegraphics[width=0.7\textwidth]{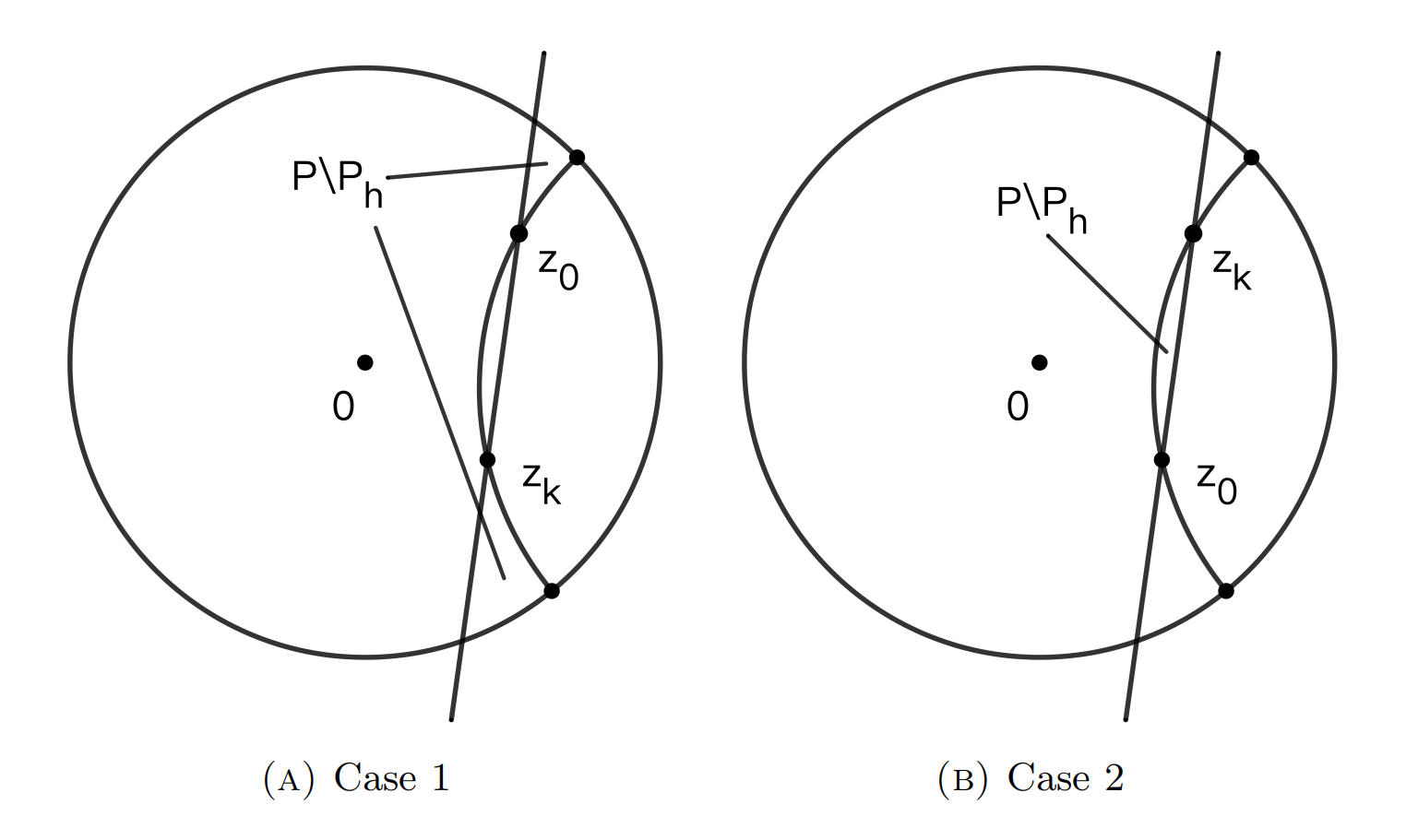}
\end{figure}
We first show the claim (1). Since $\phi$ is a geodesic embedding into $D$, we have $z_{k+1}\in P$. But if $z_{k+1}\in P\setminus P_h$, then $\phi^h$ fails to be a \hyp geodesic embedding. So we need to rule out this possibility.

Let $\ga_h$ be the entire geodesic connecting $z_0$ and $z_k$ with respect to the \hyp metric. If $\ga_h$ is a straight line, then $z_{k+1}\in P_h$. Suppose $\ga_h$ is a circle, which is orthogonal to the boundary of the unit disc $D$. We denote $z_*$ as its Euclidean center and $R$ as its Euclidean radius. Then
$$R^2+1=|z_*|^2\leq (|z_0|+R)^2.$$
So
$$1-|z_0|^2\leq 2R|z_0|\leq 2R.$$
Then
$$\sin\frac{\angle z_0z_*z_{k}}{2}=\frac{l_{ij_k}}{2R}<\frac{(1-|z_0|^2)\sin \ep}{2R}\leq \sin\ep.$$
So
$$\angle z_0z_{k+1}z_k\geq \ep>\frac{1}{2}\angle z_0z_{*}z_{k},$$
and
$$\angle z_0z_{k+1}z_k\leq \pi-\ep<\pi-\frac{1}{2}\angle z_0z_{*}z_{k}.$$
Then from the knowledge of the plane geometry, we see $z_k\in P_h$. We finish the proof of claim (\ref{1}).

Next we show the claim (\ref{2}). We have
\begin{equation}\label{3}
\arg\big(\frac{v(z_{k+1})}{v(z_k)}\big)+\arg\big(\frac{v(z_{k})}{z_k-z_0}\big)=
\arg\big(\frac{z_{k+1}-z_0}{z_k-z_0}\big)+\arg\big(\frac{v(z_{k+1})}{z_{k+1}-z_0}\big)+2n\pi,
\end{equation}
for some integer $n$. Notice that $$\arg\big(\frac{v(z_{k})}{z_k-z_0}\big)\in(-\frac{\pi}{2},\frac{\pi}{2}).$$
Then together with the claim (\ref{1}), both the right hand side and the left hand side of the Equation (\ref{3}) lies in $(-\frac{\pi}{2},\frac{3\pi}{2})$, so $n=0$. Then we obtain
$$\sum\limits_{k=1}^{m}\arg\big(\frac{v(z_{k+1})}{v(z_k)}\big)=\sum\limits_{k=1}^{m}\arg\big(\frac{z_{k+1}-z_0}{z_k-z_0}\big)=2\pi.$$
We finish the proof.
\end{proof}
\begin{rem}
In \cite{Wu22}, Wu showed that the acute condition also ensures that $\phi^h$ is a \hyp geodesic embedding.
\end{rem}
The discrete conformality and the \hyp discrete conformality are related as follows. One may refer to Lemma 5.1 in \cite{Wu22} for the proof.
\begin{lem}\label{Ehconf}
Let $\phi,\phi\pr:|T|\ra D\subset\mb{C}$ be two geodesic embeddings with the induced PL metric $l,l^{\prime}$ respectively. Suppose both $\phi,\phi\pr$ induce \hyp geodesic embeddings $\phi^h,\phi^{h\prime}:|T|\ra\mb{D}$ with the induced \hyp PL metric $l^h,l^{h\prime}$ respectively. Then $l$ and $l\pr$ are discrete conformal $l\pr=u*l$ with the conformal factor $u$ if and only if $l$ and $l\pr$ are \hyp discrete conformal $l^{h\prime}=u^h*^hl^h$ with the \hyp conformal factor $u^h$, where $u$ and $u^h$ are related by $$u^h_i=u_i+\log\frac{1-|z_i|^2}{1-|z_i\pr|^2},$$ for $z_i=\phi(i)$, $z_i\pr=\phi\pr(i)$.
\end{lem}
\begin{rem}
In fact the assumption that $\phi,\phi\pr$ induce \hyp geodesic embeddings is unnecessary if we consider the generalized (hyperbolic) PL metric.
\end{rem}

\section{Maximum principles}\label{maximum}
Maximum principle plays a very important role in partial differential equations and geometric analysis. For the discrete conformal geometry, the curvature $K$, which is clearly nonlinear as an operator on functions on the set of vertices, also satisfies the maximum principle. The following lemma is a corollary of Theorem 3.1 in \cite{LSW20}. In \cite{DGM22}, there is another proof of the maximum principle for a special case.
\begin{lem}\label{mp}
Let $i\in V^o$ and $R_i$ be its $1$-ring neighborhood. Let $\phi,\phi\pr:R_i\ra \mb{C}$ be two Delaunay geodesic embeddings with induced PL metric $l,l\pr$ respectively. Suppose $\phi$ and $\phi\pr$ are discrete conformal, $l\pr=u*l$. Then
$$\max\limits_{j\in R_i}|u_j|\leq \max\limits_{j\in \partial R_i}|u_j|.$$
Furthermore if $|u_i|=\max\limits_{j\in R_i}|u_j|$, then $u$ is a constant on $R_i$.
\end{lem}
As a direct corollary, we have the following maximum principle.
\begin{lem}\label{mpg}
Let $T=(V,E,F)$ be a triangulation of a closed surface with boundary. Let $\phi,\phi\pr:|T|\ra \mb{C}$ be two Delaunay geodesic embeddings with induced PL metric $l,l\pr$ respectively. Suppose $\phi$ and $\phi\pr$ are discrete conformal, $l\pr=u*l$. Then
$$\max\limits_{j\in V}|u_j|\leq \max\limits_{j\in \partial V}|u_j|.$$
Furthermore if $\max\limits_{j\in V^o}|u_j|=\max\limits_{j\in V}|u_j|$, then $u$ is a constant on $V$.
\end{lem}
For the \hyp setting, the minimum principle holds, even though we still call it ``maximum principle". The proof is in fact from Lemma 5.3 in \cite{Wu22}. We give a proof here to mention that the uniformly acute condition in Lemma 5.3 in \cite{Wu22} is just to ensure the \hyp geodesic embedding.
\begin{lem}\label{hmp}
Let $i\in V^o$ and $R_i$ be its $1$-ring neighborhood. Let $\phi^h,\phi^{h\prime}:R_i\ra \mb{D}$ be two Delaunay \hyp geodesic embeddings with induced \hyp PL metric $l^h,l^{h\prime}$ respectively. Suppose $\phi^h$ and $\phi^{h\prime}$ are \hyp discrete conformal, $l^{h\prime}=u^h*^hl^h$. Then
$$\min\limits_{j\in R_i}\{u^h_j,0\}\geq \min\limits_{j\in \partial R_i}\{u^h_j,0\}.$$
Furthermore if $u^h_i=\min\limits_{j\in R_i}\{u^h_j,0\}$, then $u^h$ is identically $0$ on $R_i$.
\end{lem}
\begin{proof}
Since the \hyp discrete conformality and conformal factor are invariant under the \hyp isometric group, we may assume $\phi^h(i)=0$, $\phi^{h\prime}(i)=0$. Then $\phi,\phi\pr$ induce geodesic embeddings into $D\subset\mb{C}$ with PL metric $l,l\pr$ respectively. Notice that the \hyp isometric group preserves circles, so $l,l\pr$ are also Delaunay. From Lemma \ref{Ehconf}, since $l^{h\prime}=u^h*^hl^h$, we have $l$ and $l\pr$ are also discrete conformal, $l\pr=u*l$ for $$u_j=u^h_j-\log\frac{1-|z_j|^2}{1-|z_j\pr|^2}.$$
In particular $u_i=u^h_i$. From the maximum principle Lemma \ref{mp}, we obtain there exists $j_0\in \partial R_i$ such that $u_{j_0}\leq u_i$. Suppose $u^h_i\leq 0$. Then $$|z\pr_{j_0}|=l\pr_{ij_0}=e^{(u_i+u_{j_0})/2}l_{ij_0}\leq l_{ij_0}=|z_{j_0}|.$$
Therefore
$$u_{j_0}^h=u_{j_0}+\log\frac{1-|z_{j_0}|^2}{1-|z_{j_0}\pr|^2}\leq u_i+\log\frac{1-|z_{j_0}|^2}{1-|z_{j_0}\pr|^2}\leq u_i=u_i^h.$$
So
$$\min\limits_{j\in R_i}\{u^h_j,0\}\geq \min\limits_{j\in \partial R_i}\{u^h_j,0\}.$$
Furthermore if $u^h_i=\min\limits_{j\in R_i}\{u^h_j,0\}$, then $u^h_i=u_i\leq 0$ and $u^h_i=u^h_{j_0}$ for some $j_0\in\partial R_i$. Keep track the inequalities in the discussion above, from Lemma \ref{mp} we see $u\equiv 0$ and then $u^h\equiv 0$ on $R_i$.
\end{proof}
As a direct corollary, we have
\begin{lem}
Let $T=(V,E,F)$ be a triangulation of a closed surface with boundary. Let $\phi^h,\phi^{h\prime}:|T|\ra \mb{D}$ be two Delaunay \hyp geodesic embeddings with induced \hyp PL metric $l^h,l^{h\prime}$ respectively. Suppose $\phi^h$ and $\phi^{h\prime}$ are \hyp discrete conformal, $l^{h\prime}=u^h*^hl^h$. Then
$$\min\limits_{j\in V}\{u^h_j,0\}\geq \min\limits_{j\in \partial V}\{u^h_j,0\}.$$
Furthermore if $\min\limits_{j\in V^o}\{u^h_j,0\}=\min\limits_{j\in V}\{u^h_j,0\}$, then $u^h$ is identically $0$ on $V$.
\end{lem}

\section{Modification of Wu's Proof}\label{modification}
In \cite{Wu22}, Wu showed Theorem \ref{main acute} under the uniformly acute condition. Wu followed the strategy of He's proof in his celebrated work \cite{He99}, where he showed the rigidity of infinite disc patterns. The proof has two steps, first to show the conformal factor $u$ is bounded, second to show $u$ is a constant.

In Section 3.1 of \cite{Wu22}, Wu showed the rigidity under the assumption of the boundedness of $u$, which is the following lemma.
\begin{lem}\label{bounded to rigid}
Let $T$ be a triangulated plane. Let $\phi:|T|\ra \mb{C}$ be a geodesic homeomorphism with the induced PL metric $l$. Suppose $l$ satisfies the uniformly nondegenerate condition and the uniformly Delaunay condition. Let $\phi\pr:|T|\ra \mb{C}$ be a geodesic embedding with the induced PL metric $l\pr$. Suppose $l\pr$ satisfies the Delaunay condition. If $l$ and $l\pr$ are discrete conformal, $l\pr=u*l$, with bounded conformal factor $u$, then $u$ is a constant.
\end{lem}
Notice that in Section 3.1 of \cite{Wu22}, the uniformly acute condition only plays the role to ensure that a small perturbation of $l$ preserves the Delaunay condition. So the uniformly acute condition can be replaced by the uniformly Delaunay condition in Lemma \ref{bounded to rigid}.

To show $u$ is bounded, the following lemma plays a key role. In \cite{Wu22}, it is Lemma 2.9 except the uniformly acute condition is replaced by the uniformly Delaunay condition.
For $r>0$, denote $D_r=\{z\in\mb{C}:|z|<r\}$ be the open disc of radius $r$.
\begin{lem}\label{key estimate}
Let $T=(V,E,F)$ be a triangulated plane. Let $\phi,\phi\pr:|T|\ra \mb{C}$ be two geodesic embeddings with the induced PL metric $l,l\pr$ respectively. Suppose both $l,l\pr$ satisfy the uniformly nondegenerate condition with constant $\ep>0$ and the Delaunay condition.
Suppose $l$ and $l\pr$ are discrete conformal, $l\pr=u*l$.
Let $r,r\pr>0$. Let $T_0$ be a subcomplex of $T$. Suppose $$\phi(|T_0|)\subset D_r,\quad D_{r\pr}\subset \phi\pr(|T_0|).$$ Then there is a constant $M=M(\ep)>0$ such that for every $i\in V_0$ satisfying $\phi\pr(i)\in D_{r\pr/2}$, we have
$$u_i\geq \log\frac{r\pr}{r}-M.$$
\end{lem}
We follow the idea of Wu's proof in Section 5 of \cite{Wu22}, and modify the proof by an observation that the condition in Lemma \ref{hyp emb} also implies the \hyp geodesic embedding. Before the proof of Lemma \ref{key estimate}, we show some basic estimates under the uniformly degenerate condition.
\begin{lem}\label{grad estimate}
For a triangle $\tri ijk$ satisfying the uniformly nondegenerate condition with constant $\ep>0$, then the length ratio is uniformly bounded, $\sin\ep\leq \frac{l_{ij}}{l_{ik}}\leq \sin^{-1}\ep$. As a corollary, for two PL metric on a triangle $ijk$ satisfying the uniformly nondegenerate condition with constant $\ep>0$, which are discrete conformal with conformal factor $u$, then $u$ has the gradient estimate $|u_i-u_j|\leq 4\log\sin\ep.$
\end{lem}
\begin{proof}
For a triangle $\tri ijk$, by the sine law, we have $\frac{l_{ij}}{l_{ik}}=\frac{\sin \angle\theta^{k}_{ij}}{\sin \angle\theta^{j}_{ik}}\geq \sin \angle\theta^{k}_{ij}\geq \sin\ep.$ Let $l,l\pr$ be two PL metric on a triangle $ijk$ satisfying the uniformly nondegenerate condition with constant $\ep>0$, and $l\pr=u*l$. Then $$e^{(u_i-u_j)/2}=\frac{e^{(u_i+u_k)/2}}{e^{(u_j+u_k)/2}}=\frac{l\pr_{ik}}{l_{ik}}\frac{l_{jk}}{l\pr_{jk}}\leq \sin^2\ep.$$ So
$u_i-u_j\leq 4\log\sin\ep.$
\end{proof}
Now we prove the key estimate Lemma \ref{key estimate}.
\begin{proof}
By approaching $D_{r\pr}$ by $D_{r\pr-\delta}$, we may assume $\overline{D_{r\pr}}\subset \phi\pr(|T_0|)$. By scaling, we may assume $r=\frac{\sin^3\ep}{4}\leq \frac{1}{4}$ and $r\pr=1$.
Consider $V_1=\{i\in V:\phi\pr(i)\in D=D_1\}$ and $T_1$ as the subcomplex generated by $V_1$. Then $\phi,\phi\pr$ map $|T_1|$ into $D$.
We notice that $V_1$ is finite.
In fact if $V_1$ is infinite, then there exists $\{v_n\}_{n=1,2,\cdots}\in V_1$ such that $v_n\ra v_{\infty}\in \overline{D}$.
Since $v_{\infty}$ lies in a triangle and the degree of a vertex is finite, it gives a contradiction. Let $z_i=\phi(i)$, $z_i\pr=\phi\pr(i)$.
Denote $$u^h_i=u_i+\log\frac{1-|z_i|^2}{1-|z_i\pr|^2}.$$

We claim $u_i^h\geq 0$ for every $i\in V_1$. Let $i_0$ attain the minimum of $u$ in $V_1$.\\
(1) If $i_0\in V_1^o$ and $l\pr_{i_0j}<(1-|z\pr_{i_0}|^2)\sin\ep$, for every $i_0j\in E$, then from Lemma \ref{hyp emb}, $\phi\pr$ induces a \hyp geodesic embedding $\phi^{h\prime}$ from the $1$-ring neighborhood $R_{i_0}$ of $i_0$ into $\mb{D}$. Since $\phi\pr$ is Delaunay, $\phi^{h\prime}$ is also Delaunay. Since $\phi(|V_1|)\subset D_{\sin^3\ep/4}$, for the same reason $\phi$ also induces a Delaunay \hyp geodesic embedding $\phi^h$ from $R_{i_0}$ into $\mb{D}$. Then the hyperbolic maximum principle Lemma \ref{hmp} implies $u_{i_0}^h\geq 0$.\\
(2) If $i_0\in V_1^o$ and there exists $j_0\in V_1$, $i_0j_0\in E$ such that $l\pr_{i_0j_0}\geq(1-|z\pr_{i_0}|^2)\sin\ep$, then from Lemma \ref{grad estimate}
\begin{\eq}
e^{u_{i_0}^h}&=&e^{u_{i_0}}\cdot\frac{1-|z_{i_0}|^2}{1-|z_{i_0}\pr|^2}\\
&=&\frac{l\pr_{i_0j_0}}{l_{i_0j_0}}\cdot e^{(u_{i_0}-u_{j_0})/2}\cdot \frac{1-|z_{i_0}|^2}{1-|z_{i_0}\pr|^2}\\
&\geq&\frac{l\pr_{i_0j_0}}{l_{i_0j_0}}\cdot\sin^2\ep\cdot\frac{1-|z_{i_0}|^2}{1-|z_{i_0}\pr|^2}\\
&\geq&\sin^3\ep\cdot\frac{1-|z_{i_0}|^2}{l_{i_0j_0}}\\
&\geq&\sin^3\ep\cdot\frac{1/2}{2r}=1.
\end{\eq}
(3) If $i_0\in\partial V_1$, then there exists $j_1\in V$ such that $j_1\notin D$.
Then $l\pr_{i_0j_1}\geq1-|z\pr_{i_0}|.$
As the estimates above, assume $\sin\ep\leq \frac{1}{2}$, we have $$e^{u_{i_0}^h}\geq\frac{l\pr_{i_0j_1}}{l_{i_0j_1}}\cdot\sin^2\ep\cdot\frac{1-|z_{i_0}|^2}{1-|z_{i_0}\pr|^2}
\geq\frac{1}{l_{i_0j_1}}\cdot\sin^2\ep\cdot\frac{1-|z_{i_0}|^2}{1+|z_{i_0}\pr|}\geq\sin^3\ep\cdot\frac{1/2}{2r}= 1.$$

So we show $u^h_i\geq 0$ for every $i\in V_1$. Then for $i\in V_0$ satisfying $\phi\pr(i)\in D_{r\pr/2}$, set $M=-\log\frac{\sin^3\ep}{2}$, we have
$$u_i=u_i^h-\log\frac{1-|z_i|^2}{1-|z_i\pr|^2}\geq -\log\frac{1-|z_i|^2}{1-|z_i\pr|^2}\geq \log (1-|z_i\pr|^2)\geq\log\frac{1}{2}=\log\frac{r\pr}{r}-M.$$
\end{proof}
By using the estimate Lemma \ref{key estimate} and the conformal modulus, Wu showed the boundedness of $u$ in Section 3.2 of \cite{Wu22}, which is the following lemma.
\begin{lem}\label{boundedness}
Let $T$ be a triangulated plane. Let $\phi:|T|\ra \mb{C}$ be a geodesic homeomorphism with the induced PL metric $l$. Let $\phi\pr:|T|\ra \mb{C}$ be a geodesic embedding with the induced PL metric $l\pr$. Suppose both $l,l\pr$ satisfy the uniformly nondegenerate condition with constant $\ep>0$ and the Delaunay condition.
Suppose $l$ and $l\pr$ are discrete conformal, $l\pr=u*l$. Then $u$ is bounded.
\end{lem}
Together with Lemma \ref{boundedness} and Lemma \ref{bounded to rigid}, we finish the proof of Theorem \ref{main Delaunay}.

\end{document}